\documentclass{article}
\usepackage[T2A]{fontenc}
\usepackage[cp1251]{inputenc}
\usepackage{amsmath,amssymb,amsthm,amsfonts,amscd}
\usepackage[all]{xy}

\newtheorem{lemma}{Lemma}
\newtheorem{remark}{Remark}
\newtheorem{theorem}{Theorem}
\newtheorem{definition}{Definition}
\newtheorem{corollary}{Corollary}

\large

\def\R{{\Bbb R}}

\def\Z{{\mathbb{ Z}}}

\def\RP{{\Bbb R}\!{\rm P}}

\def\C{{\mathbb{C}}}
\def\H{{\mathbb{H}}}
\def\Q{{\bf Q}}

\def\k{{\bf k}}

\def\i{{\bf i}}
\def\j{{\bf j}}

\def\Zt{{\Z^{tw}}}

\def\mod{{\operatorname{mod}\,}}

\def\R{{\Bbb R}}

\def\e{{\bf e}}

\def\Imm{{\rm Imm}}

\def\e{{\varepsilon}}

\begin{document}
\title{Local coefficients and the Herbert Formula}
\author{Petr M. Akhmet'ev\footnote{National Research University Higher School of
Economics, Moscow, Russia; 108840, IZMIRAN, Troitsk, Moscow region, Russia}
 and Theodore Yu. Popelenskii
\footnote{Moscow State University, Faculty of Mechanics and Mathematics, Leninskie Gory 1, Moscow, 119991 Russia}}
\date{}
\maketitle
\medskip

\section{Introduction}

{\it The Herbert Theorem} is used in papers on immersions theory.  
Let us recall ths theorem. Consider smooth closed manifold $N$ of dimension  $n-k$ and consider its immersion
$g:N^{n-k}\looparrowright  \R^n$, which is self-transversal. Then for a manifold
$$
\bar L = \{(x,y)\in N\times N\;|\; x\ne y, g(x)=g(y)\},
$$
which is called a self-intersection manifold of the immersion $g$,
the composition
$j: \bar L \subset N\times N \stackrel{pr_1}{\to} N$ is an immersion.
Consider the normal bundle $\nu(g)$ of the immersion  $g$.
Let us assume that $N$ is oriented. The bundle
$\nu(g)$  and the manifold $\bar L$ is naturally   oriented. 
The homology class, dual in $N$ to the Euler class of the normal bundle
$\nu(g)$, is called  {\it the homology Euler class} and is denoted by  $e_*(\nu(g))$.
Then in the group  $H_{n-2k}(N,\Z)$  the following  formula is satisfied ({\it the Herbert formula}):
\begin{equation}\label{eq:herbert}
j_*[\bar L] + e_*(\nu(g))=0. 
\end{equation}

In the case the manifold $N$ is non-oriented, the  formula ~(\ref{eq:herbert}) is 
satisfied, if  homology and cohomology groups is defined over $\Z/2$-coefficients.

Let us note that the original Herbert's paper~\cite{herbert} contains formulas of  $k$-uple self-intersection points of immersions. Here are several generalizations of the Herbert formula, 
and several approaches to prove it, see 
~\cite{eccles-grant, lippner-szucs,  Grant}. Also self-intersections of immersions 
are investigated using homology of classifying loop-spaces, see \cite{Eccles}.

In the present paper an  interest to immersions is related with the classical 
paper ~\cite{Hirsch-paper}, and with  relationships  to homotopy groups of spheres and of 
Thom spaces, see~\cite{pontr, Wells}. Let us mention papers~\cite{A-E,akhmetiev-frolkina, A-F}
by one of the two authors.

In this paper we discuss  a generalisation of the Herbert formula for double points, when the
normal bundle of the immersion $g$ admits an additional structure, and an application.

Below the Whitney sum of $k$ copies of a prescribed vector-bundle $\xi$ is denoted by  $k\xi$.
The trivial line (1-dimensional real) bundle over $N$ is denoted by  $\e_N$.


\section{Immersions with additional structure of normal bundles}

Let us define a number of regular immersions cobordism groups.
Assume a group $G$ and its representation in the space $\R^d$ are fixed.
The following vector bundle  
$\gamma=EG\times_G \R^d \to BG$ is well-defined.  
\begin{definition}\label{def:G-str}
 Let us say an immersion  
 $g:N^{n-kd}\looparrowright \R^n$ is equipped by a  {\em $G$-structure} (of its normal bundle), 
if the following data are fixed:

(a) a continous mapping $\eta: N\to BG$;

(b) a normal bundles isomorphism  $ \Xi:\nu(g)\to \oplus k\, \eta^*(\gamma)$. 
\end{definition}

Strictly say, to define
 <<$G$-structure>>, a representation of $G$ has to be prescribed. Or, as in~\cite{akhmetiev-frolkina},
a universal bundle  $\gamma$ should be fixed. In our case for a group $G$ the only
representation is considered, this give no collisions.

Let us consider arbitrary triples
 $(g:N^{n-kd}\looparrowright\R^n,\eta, \Xi)$, where $N$ is closed.


\begin{definition} \label{def:bord} 
Let us say that triples
 $(g_0\colon N_0^{n-kd}\to \R^n,\eta_0, \Xi_0)$ and
$(g_1\colon N_1^{n-kd}\to \R^n,\eta_1, \Xi_1)$ are  {\em cobordant}, if the following 3 conditions are satisfied: 

(i) there exists a compact manifold  $W^{n-kd+1}$ and a diffeomorphism  
$\varphi_0\sqcup\varphi_1:N_0\sqcup N_1 \to \partial W$;


(ii) there exists an immersion  $G:W^{n-kd+1}\looparrowright \R^{n}\times[0,1]$, 
which is orthohonal along the boundary
 $\partial W$ to 
$\R^{n}\times\{0,1\}$ and which is agree with restrictions of the immersions $g_0$ � $g_1$ to  $\partial W$,  namely, for $s=0,1$ the restriction  $G|_{\varphi_s(N_s)} $ coinsids with the composition
 $N_s\stackrel{\varphi_s}{\rightarrow}\R^n\to \R^n\times\{s\}$;

(iii) there exists a continous mapping  $\eta:W^{n-kd+1}\to BG$ and an isomorphism 
$\Xi:\nu(g)\to \oplus k\, \eta^* (\gamma) $, that are agree with 
 $\eta_s$ and $\Xi_s$ ($s=0,1$) on the boundary of   $W$ 
in the following natural sence:
 $\eta|_{\varphi_s(N_s)} = \eta_s$ and 
$\Xi|_{N_s}=\Xi_s$. 
\end{definition}

By standard arguments one could proves that the cobordism relation, introduced above, is an equivalence
relation. A disjoin union induces an Abelean group structure on the set of cobordism classes, 
this group is denoted by
  $\Imm^G(n-kd,kd)$. More detailed definition is in~\cite{akhmetiev-frolkina}.

In the case when  
$G$-bundle  $\gamma$ is oriented, the manifold   $N$ is equipped with  a natural 
orientation, constructed by $\Xi$. In a general case the bundle $\gamma$
is non-oriented, and 
$N$ is equipped by a natural orientation only in the case $k \equiv 0 \pmod{2}$. 
Nevertheless, for a non-oriented  $\gamma$ and for  $k \equiv 1 \pmod{2}$ on  $N$ a local orientation is well-defined. Let us consider the details for a special case
 $G \cong \Z_4 \rtimes \Z$.

Let $G$ be a group, equipped with an an action of $\Z$ by a (non-trivial) authomorphism
$\theta=\theta(1): G \mapsto G$, $1 \in \Z$.
The classifying space   $BG\rtimes \Z$ is defined as a semi-direct product of $BG$  and the standard circle $S^1$, this space is the factor of $BG\times [0,1]$ by identification $BG\times\{0\}$ and $BG\times\{1\}$ using the homeomorphism  $B\theta$. The homomorphism $G\rtimes \Z\to \Z$ classifies the projection $pr_S:BG\rtimes \Z\to S^1$, 
this projection is a locally trivial fibration with the fibre  $BG$. 

Let us fix an embedding  $T=\dot{D}^{n-1}\times S^1\subset \R^n$
of an open solid torus onto the standard solid torus, which is invariant by rotation trough the
coordinate axis.
The projection 
 $pr_T$  of the solid torus  $T$  onto the axial circle $S^1$ is well-defined. Denote by$pt \in S^1$ a marked point on the circle. 
Thus,  $pr_S^{-1}(pt) \subset BG \rtimes \Z$ coincides with 
$BG \subset BG \rtimes \Z$. 

Let us prescribe a representation $G\rtimes \Z$ � $\R^d$ and the corresponding  vector $d$-bundle $\gamma$ over $BG\rtimes \Z$ (the universal bundle).

\begin{definition} 
Let us say that for an immersion
 $   g:N^{n-kd}\looparrowright \R^n$  {\em a  $G\rtimes\Z$-structure} (of the normal bundle of $g$) is fixed, is the following conditions are satisfied:

(a) the immersion $g$ is an immersion into the solid torus $T \subset \R^n$, i.e. is represented by a composition 
$N^{n-kd}\stackrel{\tilde g}{\looparrowright} T {\subset} \R^n$;

Denote by $K^{n-kd-1} \subset N^{n-kd}$  a codimension $1$ submanifold, which is given by the formula: 
$(pr_T \circ g)^{-1}(pt)$, assuming, that $pt$ is a regular value for  $pr_T \circ g$.
Denote by $U_K \subset N^{n-kd}$ a thin regular neighbourhood of the submanifold  $K^{n-kd-1} \subset N^{n-kd}$.

(b) a continuous mapping  $\eta: N\to BG\rtimes \Z$ , which induces the corresponding mapping 
of pairs
$\eta: (N; U_N, N\setminus U_N) \to (BG \rtimes \Z; BG, BG \rtimes \Z \setminus BG)$; moreover, the composition  $pr_S\circ \eta$ is homotopic to  a composition $pr_T\circ\tilde g$, and the homotopy class of this homotopy is given; 

(�) an isomorphism of vector bundles  $ \Xi:\nu(g)\to \oplus k\, \eta^*(\gamma)$ is well-defined. 
\end{definition}

For such mappings with prescribed   $G\rtimes\Z$-framing of normal bundle
the cobordism relation is well-defined, the definition is analogous to~\ref{def:bord} and omitted.
On the cobordism classes of $G\rtimes\Z$-framed immersions Abelian group structure is well-defined:
the sum of two immersions is defined as the disjoin union of the two parallel copies of immersions in $T$.

Denote this cobordism group by $\Imm^{G\rtimes\Z}(n-kd,kd;T)$.

\subsection*{ Cobordism group $\Imm^{\Z_4 \rtimes \Z}(n-2k,2k;T)$}

Consider the group $\Z_4\rtimes \Z$. This group is generated by elements  $b\in\Z_4$, $\tilde a\in\Z$ 
with the relations 
$b^4=e$,  $\tilde a^{-1}b\tilde a=b^3$. A representation $A$ of the group $\Z_4\rtimes \Z$
in $\R^2$ is given by the formulas: 
$$
b \mapsto \left(\begin{matrix}0&-1\\1&0\end{matrix}\right),\ \ \ 
\tilde a \mapsto
 \left(\begin{matrix}0&1\\1&0\end{matrix}\right).$$
Thus, one may consider the group 
  $\Imm_T^{\Z_4 \rtimes \Z}(n-2k,2k;T)$, which is gepresented by immersions $g:N^{n-2k}\looparrowright T\subset \R^n$  with $\Z_4\rtimes \Z$-framings.

If $k$ is even, the bundle $\nu(g)$  is oriented  (because $w_1(\nu(g)) = k w_1(\eta^*(\gamma))=0$ even for an non-oriented $\eta^*(\gamma)$). Then  $N$  and $\bar L$ are oriented and for prescribed orientations the Herbert formula is satisfied:
\begin{equation*}
j_*[\bar L] + e_*(\nu(g))=0  \hbox{ in the group } H_{n-4k}(N,\Z).
\end{equation*}

In the case $k$ is odd, the considered manifolds, generally speaking, are non-oriented.
But, the Herbert formula is generalized using oriented local coefficient system.
 
Denote on the space $B\Z_4\rtimes\Z$ the local coefficients system  $\Z^{tw}$ of
$\pi_1(B\Z_4\rtimes\Z)=\Z_4\rtimes\Z$ by authomorphisms of $\Z$ by the following formula:
$$
\Z_4\rtimes\Z \stackrel{pr_\Z}{\longrightarrow} \Z \stackrel{\mod 2}{\longrightarrow} \Z_2 =\mathrm{Aut}(\Z).
$$
Analogously, the corresponding local system is well-defined for an arbitrary $N^{n-kd}$,
which is included into a triple
 $(g,\eta,\Xi)$, representing an element in $\Imm^{\Z_4\rtimes\Z}(n-kd,kd;T)$. To define the local 
system the pull-back $\eta^*(\Z^{tw})$ is used. 
This local system on $N$ is also denoted by $\Z^{tw}$, this gives no confusion. 

In this situation, let us discuss terms in the Herbert formula.
Let $(g:N^{n-2k}\looparrowright T\subset \R^n, \eta : N \to B\Z_4\rtimes \Z,\Xi)$ be a
representation of an element in $\Imm^{\Z_4\rtimes \Z}(n-2k,2k;T)$. 
Double-point manifold of $\bar L$ is oriented. This follows from the formula:
$\nu (\bar L\looparrowright \R^n) = \nu (\bar N\looparrowright \R^n)|_{\bar L}\oplus \nu (\bar L\looparrowright N)=
\nu (g)|_{\bar L}\oplus \tau^*\nu (g)|_{\bar L}$,
where $\tau:\bar L\to \bar L$ is the standard involution, which is induced by permutation of coordinates in    $N\times N$. The oriented classes $w_1$ of the bundles $\nu (g)|_{\bar L}$ and $ \tau^*\nu (g)|_{\bar L}$
coincide, because a closed loop  
 $\gamma$ in $\bar L$  changes the orientation of the fibre   $\nu (g)|_{\bar L}$, iff
the loop $\tau\circ \gamma$  also changes the orientation. Thus, one may define the fundamental class
 $[\bar L]\in H_{n-4k}(\bar L,\Z)$. This proves that the first term $i_{\bar L\looparrowright N}([\bar L]) \in H_{n-4k}(N,\Z)$ in the formula is well-defined. 

Let us consider the second term. The main difficult concerns to non-orientable bundles.
But in the considered case this problem is solved. Denote by  $o(\nu(g))$ the orienting 
system (the  locally coefficients system with coefficients $\Z$) of the bundle  
$o(\nu(g))$. Recall, that on the manifold $N$ a local coefficient $\Z$-system is given by the corresponding element
$\mathrm{Hom}\,(\pi_1(N), \mathrm{Aut}\,(\Z))= H^1(N,\Z_2)$. The oriented system $o(\nu(g))$ is given by the class 
 $w_1(\nu(g))\in H^1(N,\Z_2) $. For a bundle   $\nu(g)$ the Euler class  $e(\nu(g))\in H^{2k}(N,o(\nu(g)))$ is well-defined. Alternatively, the manifold $N$ is oriented with the local coefficients system  
  $o(TN)$. Using the decomposition  $TN\oplus \nu(g)=n\e_N$, where $\e_N$ is the trivial line bundle over
$N$, the isomorphism  $o(TN)=o(\nu(g))$ is well-defined. Thus, the fundamental class
 $[N]\in H_{n-2k}(N,o(\nu(g))) $ is well-defined.
Obviously, $o(\nu(g))=(\eta^*(\Z^{tw}))^{\otimes k}$. Then $H_{n-2k}(N,o(\nu(g))) = H_{n-2k}(N,(\eta^*(\Z^{tw}))^{\otimes k} )$. From the last formula the Euler homology class
$e_*(\nu(g))\in H_{n-4k}(N, \Z)$ is defined by
 $[N]\cap e(\nu(g)) = [N]\cap(e(\eta^*(\gamma)))^k$.

\begin{theorem}
Assume that the manifold $N$ is equipped by  $\Z_4\rtimes \Z$-framing. Then in the group
 $ H_{n-4k}(N, \Z) $ the following formula is well-defined:
 \begin{equation}\label{eq:herbert1}
j_*[\bar L] + e_*(\nu(g))=0.
\end{equation}
\end{theorem}

\begin{proof}
The proof follows from a straightforward generalisation of classical arguments~\cite{herbert}.
\end{proof}

\begin{remark}
Obviously, the theorem is generalised for  
 $G\rtimes Z$-framing, in the case of representation  $G\to SO(d)$, which is extended to a representation 
 $G\rtimes\Z \to O(d)$
\end{remark}

\subsection*{Hurewicz homomorphim}
Define the Hurewicz homomorphism $$H: \Imm^{\Z/4\rtimes \Z}(n-2k,2k;T) \to H_{n-2k}(B\Z/4\rtimes\Z;(\Z^{tw})^{\otimes k}).$$
Take an $x\in  \Imm^{\Z/4\rtimes \Z}(n-2k,2k;T)$, which is represented by a triple 
$(g:N^{n-2k}\looparrowright T\subset \R^n, \eta : N \to B\Z_4\rtimes \Z,\Xi)$.
The fundamental class 
$[N]\in H_{n-2k}(N;\eta^*((\Z^{tw})^{\otimes k}))$ is well-defined.
Define $H(x)=\eta_*([N])$. Obviously, this class depends not of a representation of $x$. 



\begin{lemma}\label{1}
(a) The mapping 
 $i: B\Z_4 \subset B\Z_4\rtimes\Z$ of the classified spaces, corresponded to the monomorphism  $\Z_4 \subset \Z_4\rtimes\Z$,  induces the morphism $i^*(\Z^{tw}) = \Z$ of local coefficients systems. 

(b) Assuming $q=4m+1$, the isomorphism is well-defined:
$$i_*:  \Z_4=H_{q}(B\Z/4;\Z) \cong H_{q}(B\Z/4;i^*(\Z^{tw})) \to H_{q}(B\Z/4\rtimes \Z;\Z^{tw}).$$

(c) Assuming
$q=4m+3$, the isomorphism is well-defined:
$$i_*:   \Z_4=H_{q}(B\Z/4;\Z) \to H_{q}(B\Z/4\rtimes \Z;\Z).$$

(d) $i^*:H^{2}(B\Z/4\rtimes \Z;\Z^{tw})\to H^{2}(B\Z/4;\Z)=\Z_4$
is isomorphism, and the Euler class $e=e(\gamma)$ of the universal bundle   $\gamma$ 
is a generator. Moreover, the formula $\cap e \mapsto -\cap e$ determines isomorphisms
 $H_{4m+1}(B\Z_4\rtimes Z; \Z^{tw})\to
H_{4m-1}(B\Z_4\rtimes Z; \Z)$ and $H_{4m-1}(B\Z_4\rtimes \Z; \Z)\to H_{4m-3}(B\Z_4\rtimes \Z; \Z^{tw})$
 for $m\ge 1$.
\end{lemma}

\begin{proof}
Recall, that  $H_{q}(B\Z/4;\Z)=\Z_4$ for an arbitrary odd $q\ge 1$, and $H_{0}(B\Z/4;\Z)=\Z$.
Homology of the group $\Z$ in a module  $M$ are known: $H_0(B\Z,M)=M_{\Z}$ 
(coinvariants of  the $\Z$-action  on $M$)
and $H_1(\Z,M)=M^{\Z}$ (invariants of  the $\Z$-action  on $M$).

Consider the space $E\Z_4\rtimes\Z$, this space is realized as $S^\infty\times\R$.
Here we have $S^\infty = \bigcup\limits_n\{(z_1,\ldots, z_n)\in \C^n\;|\; \sum|z_k|^2=1\}$.
The action of the generators $b,\tilde a$ of the group $\Z_4\rtimes\Z$ are given by the formulas:
 $b(z,t)=(e^{\pi i/2}z,t)$, and
$\tilde a(z,t)=(\bar z,t+1)$. It is easy to see that the relations $b^4=1$, and $\tilde a^{-1}b \tilde a=b^3$ are satisfied, the corresponding (free) action of   $\Z_4\rtimes\Z$ on
  $E\Z_4\rtimes\Z=S^\infty\times\R$ is well-defined.

Consider the spectral sequence of the bundle	 $B\Z_4\to B\Z_4\rtimes\Z\to B\Z$ 
for the trivial  module  $\Z$ of coefficients. The second term of this sequence 
$E^2_{pq  }= H_p(B\Z, H_q(B\Z_4,\Z))$. By dimension reason,
  $E^2_{**}=E^\infty_{**}$. Non-trivial terms could be only 
$E^2_{0q}$ and $E^2_{1,q-1}$ for odd
 $q\ge 1$ and for  $q=0$. 
Therefore, the problem is to calculate the action of the generator 
 $\tilde a\in \Z$ in homology  $H_q(B\Z_4,\Z)$.

Denote the corresponding homomorphism by $\tilde a_*$. Because non-trivial groups $H_q(B\Z_4,\Z)$  are isomorphic to $\Z$, or to $\Z_4$, the action $\tilde a_*$ in homologies  $H_q(B\Z_4,\Z)$ is the multiplication on  $\pm 1$. The action of the generator $\tilde a $  coincides to the conjugation of all coordinates in $S^\infty$, therefore the sign depends of   $q$ as follows:  $\tilde a_*=1$ for $q=4m+3$, and  $\tilde a_*=-1$ for $q=4m+1$.
Therefore $E^2_{0q}$ and $E^2_{1q}$  are isomorphic to  $\Z_4$ for $q=4m+3$ and to  $\Z_2$ for $q=4m+1$.
Statement (a) is proved.

Calculation of cohomologies $H_*(B\Z_4\rtimes\Z,\Zt)$  are analogous. 
Let us note, that to calculate the  action of $\Z$ in homologies $H_q(B\Z_4,i^*(\Zt))$ an additional sign  $-1$ is required, because the generator $\tilde a$ acts in the module of the coefficients
by multiplication on   $-1$. Therefore, we get:
$\tilde a_*=1$ for $q=4m+1$ and  $\tilde a_*=-1$ for $q=4m+3$.
The terms $E^2_{0q}$ � $E^2_{1q}$  are isomorphic  to $\Z_4$ for $q=4m+1$ and to $\Z_2$ for $q=4m+3$.
Proofs of the last statements are analogous.

\end{proof}

\subsection*{Self-intersection manifolds of $\Z/4 \rtimes\Z$-framed immersions  }

Let $(g,\eta,\Xi)$ be a $\Z/4 \rtimes\Z$-framed immersion. 
The self-intersection manifold  $L^{n-4k}$ of the immersion $g$ is defined as
$L=\bar L/\tau$, where  $\tau$ is the coordinate involution on $N\times N$. The 2-sheeted covering  $\pi:\bar L\to L$, which is called
the canonical covering, is well-defined. The following conditions are satisfied.

(1) There exists an immersion  $h: L^{n-4k} \looparrowright  T \subset \R^n$, 
for which 2-sheeted covering  $\bar L^{n-4k}$ is included in the following diagram, \\
\centerline{\xymatrix{
\bar L \ar[r]^<<<<<{j}\ar[d]_{\pi} & N^{n-2k}\ar[r]^>>>>>{\tilde g}&T\ar@{=}[d]\\
L\ar[rr]^{h} &&T}}
In this diagram the horizontal mappings are immersions. Recall, that the immersion
 $ j:\bar L\to N$ is defined by composition of the embedding  $\bar L\subset N\times N$ 
and the projection  $N\times N\to N$  onto the first factor. 

(2) The normal bundle over $\bar L$ is decomposed into 2 factors $\nu(g)|_{\bar L} \oplus \tau*\nu(g)|_{\bar L}$, where $\tau: \bar L\to \bar L $ is the involution, which permutes the sheets of the double covering. Therefore the normal bundle over  $\bar h:\bar L\looparrowright  T$ is decomposed 
into the Whitney sum of $k$ isomorphic copies of the bundle
$\eta\oplus \tau^*\eta$, each copy is classified by a mapping in 
$B(\Z_4\rtimes \Z)\times (\Z_4\rtimes \Z)$.
Because $\tilde h$ and $\tilde h\circ \tau$ coincide by definition of the manifold 
 $\bar L$, compositions
 $\nu|_{\bar L}$ and $\tau^*\nu|_{\bar L}$ with the projection  $pr_S : B\Z_4\rtimes Z \to B\Z$ 
coincide. Therefore, the classifying mapping of the bundle
 $\eta\oplus \tau^*\eta$, is decomposed using the mapping  $B(\Z_4\times  \Z_4)\rtimes\Z \to B(\Z_4\rtimes \Z)\times (\Z_4\rtimes \Z)$, which is induced by the diagonal inclusion
 $\Z\subset \Z\times\Z$. This shows that the normal bundle over
 $h:L\looparrowright  T\subset \R^n$  is decomposed into the Whitney sum of $k$ copies of
a 4-dimensional bundle, which is classified by the mapping
  $\zeta:L\to B((\Z_4\times  \Z_4)\rtimes\Z)\rtimes \Z_2$.
The corresponding universal bundle over  $B((\Z_4\times  \Z_4)\rtimes\Z)\rtimes \Z_2$ 
will be denoted by
 $\gamma^{[2]}$ for short. Let us note, that 
  $((\Z_4\times  \Z_4)\rtimes\Z)\rtimes \Z_2=
 ((\Z_4\times  \Z_4)\rtimes\Z_2)\rtimes \Z$, because the actions of $\Z$ and $\Z_2$ on 
$\Z_4\times  \Z_4$ are commuted with each other
(a generator of  $\Z$ multiplies the both generators elements of $\Z_4\times  \Z_4$ by $-1$,  
 the generator in  $\Z_2$ changes the above generators; the action between $\Z$ and $\Z_2$ 
is trivial).

Thus,, the triple 
$(g,\eta,\Xi)$ with a $\Z_4\rtimes \Z$-framing determines a triple
 $(h,\zeta,\Lambda)$, where $h:L^{n-4k}\looparrowright \R^n$ is an 
immersion,  $\zeta:L\to B((\Z_4\times  \Z_4)\rtimes\Z)\rtimes \Z_2$ is a continous mapping
and  $\Lambda:\nu (h)\to k\zeta^*(\gamma^{[2]})$ is an isomorphism.

\subsection*{Subgroups $(\Z_4 \times \Z_4) \rtimes \Z_2$ and $((\Z_4 \times \Z_4) \rtimes \Z_2)\rtimes \Z$ and its representation in  $SO(4)$}

The representation $A^{[2]}$ of the group $((\Z_4 \times \Z_4) \rtimes \Z_2)\rtimes Z$ in $O(4)$
is given by the formulas:
\begin{align*}
&b_1=(((1,0),0),0)\mapsto \left(
\begin{array}{cccc}
0 & -1 & 0 & 0 \\
1 & 0 & 0 & 0 \\
0 & 0 & 1 & 0 \\
0 & 0 & 0 & 1 \\
\end{array}
\right),& &
b_2=(((0,1),0),0)\mapsto \left(
\begin{array}{cccc}
1 & 0 & 0 & 0 \\
0 & 1 & 0 & 0 \\
0 & 0 & 0 & -1 \\
0 & 0 & 1 & 0 \\
\end{array}
\right),\\
&t=(((0,0),1),0)\mapsto \left(
\begin{array}{cccc}
0 & 0 & 1 & 0 \\
0 & 0& 0 & 1 \\
1 & 0 & 0 & 0 \\
0 & 1 & 0 & 0 \\
\end{array}
\right),& &
\tilde a=(((0,0),0),1)\mapsto \left(
\begin{array}{cccc}
0 & 1 & 0 & 0 \\
1 & 0 & 0 & 0 \\
0 & 0 & 0 & 1 \\
0 & 0 & 1 & 0 \\
\end{array}
\right).
\end{align*}

\subsubsection*{ Subgroup $\Q \subset (\Z_4 \times \Z_4) \rtimes \Z_2$}

By $\Q$ is denoted the standard integer quaternion group, generated by  $\pm 1, \pm\i, \pm\j, \pm\k$, this group is included into
$(\Z_4 \times \Z_4) \rtimes \Z_2$ using the correspondence:
$$
1\mapsto ((0,0),0),\ \ 
\i\mapsto ((1,-1),0),  \ \ 
\j\mapsto ((1,1),1), \ \ 
\k\mapsto ((2,0),1).
$$

Then the standard representation of the group $\Z_4$  in $\R^2$ by rotations trough angles divided by $\pi/2$, induces the representation of the group
 $(\Z_4 \times \Z_4) \rtimes \Z_2$ in $\R^4$. This representation maps the generators of $\Q$
by the following formulas:

 $$
\i\mapsto\left(
\begin{array}{cccc}
0 & -1 & 0 & 0 \\
1 & 0 & 0 & 0 \\
0 & 0 & 0 & 1 \\
0 & 0 & -1 & 0 
\end{array}
\right),\ \ 
\j\mapsto\left(
\begin{array}{cccc}
0 & 0 & 0 & -1 \\
0 & 0 & 1 & 0 \\
0 & -1 & 0 & 0 \\
1 &  0 & 0 & 0 
\end{array}
\right),\ \ 
\k\mapsto\left(
\begin{array}{cccc}
0 & 0 & -1 & 0 \\
0 & 0 & 0 & -1 \\
1 & 0 & 0 & 0 \\
0 & 1 & 0 & 0 \\
\end{array}
\right).
$$

\subsubsection*{Subgroups $\Z_4 \times \Z_2\subset (\Z_4\times\Z_4) \rtimes \Z_2$ and 
 $(\Z_4 \times \Z_2)\rtimes \Z \subset ((\Z_4\times\Z_4) \rtimes \Z_2)\rtimes \Z$}

The standard diagonal mapping 
$\Z_4\to \Z_4\times \Z_4$  induces an inclusion  $\Z/4 \times \Z/2 $ onto a subgroup in 
$ (\Z/4 \times \Z/4) \rtimes \Z/2$.
The generators   $b$, $t$ of the group $\Z/4 \times \Z/2$ 
are mapped into the elements
 $((1,1),0)$ and $((0,0),1)$ of the group
$(\Z/4 \times \Z/4)\rtimes \Z/2$ therefore this elements are represented by the following 
formulas:

$$b\mapsto \left(
\begin{array}{cccc}
0 & -1 & 0 & 0 \\
1 & 0 & 0 & 0 \\
0 & 0 & 0 & -1 \\
0 & 0 & 1 & 0 \\
\end{array}
\right),\ \ 
t\mapsto \left(
\begin{array}{cccc}
0 & 0 & 1 & 0 \\
0 & 0 & 0 & 1 \\
1 & 0 & 0 & 0 \\
0 & 1 & 0 & 0 \\
\end{array}
\right).$$

The subgroup $\Z/4 \times \Z/2  \rtimes \Z \subset ((\Z/4 \times \Z/4)\rtimes \Z/2) \rtimes \Z$ 
is defined as extension of the considered subgroup by he inclusion  $(\Z/4 \times \Z/4) \rtimes \Z/2
\subset ((\Z/4 \times \Z/4) \rtimes \Z/2) \rtimes \Z$.

\subsection*{Cap-product with 2-dimensional class}
For a representation   $(g:N^{n-2k}\looparrowright \R^n, \eta,\Xi)$ of a prescribed class
in  
$\Imm^{\Z/4\rtimes \Z}(n-2k,2k;T) $ the bundle   $\eta^*(\gamma)$ is well-defined. Trivial vectors of a generic section form a submanifold   $N_1\subset N$, 
moreover, for the image of the fundamental class
$[N_1]\in H_{n-2k-2}(N_1,(\Z^{tw})^{\otimes (k+1)})$ by the embedding  $i_N:N_1\to N $ 
the following formula is satisfied:
 $(i_N)_*[N_1] = [N]\cap \eta^*(e)$ in the group $H_{n-2k-2}(N, (\Zt)^{\otimes (k+1)})$. 
Therefore, in the group   $H_{n-2k-2}(B\Z_4\rtimes\Z, (\Zt)^{\otimes(k+1)})$ the following equation is satisfied: 
$H(N_1)=H(N)\cap e$ (or, $\eta_*\circ (i_N)_*[N_1]= \eta_*[N]\cap e$).

\subsection*{Manifolds of self-intersection points of the immersions $\bar L_1\subset N_1$  and $\bar L \subset N$}

Recall, in the Introduction the manifold  $\bar L$ is defined by the formula:
$$
\bar L = \{(x,y)\in N\times N\;|\; x\ne y, g(x)=g(y)\},
$$
and the immersion
$j: \bar L \subset N\times N \stackrel{pr_1}{\to} N$ is well-defined.
There exists a 2-sheeted covering $\pi :\bar L\to L$,  where $L = \bar L/\tau$, and $\tau$ 
is the involution on $N\times N$, which permutes the factors:  $\tau(x,y)=(y,x)$.
The manifold  $L_1$ and its 2-sheeted covering  $\bar L_1\subset N_1\times N_1$ are defined analogously.
Denote by $i_L$  the inclusion of $L_1$ into $L$.

Using the Herbert Theorem for immersions $g: N^{n-2k}\to \R^n$  � $g_1:N^{n-2k-2}\to \R^{n}$,
we get the following 2 relations:

\ \ \ \ $j_*[\bar L]+[N]\cap \eta^*(e)^k=0$ � $H_{n-4k}(N,(\Zt)^{\otimes 2k})$,

\ \ \ \ $(j_1)_*[\bar L_1]+[N_1]\cap \eta^*(e)^{(k+1)}=0$ � $H_{n-4k-4}(N,(\Zt)^{\otimes(2k+2)})$.

\noindent
Let us apply the homomorphism $(i_{N})_*$ for elements of the formula, in the group 
$H_{n-4k-4}(N,(\Zt)^{2k+2})$ we get the formula:
\begin{equation}\label{eq:vazhnoe}
j_*[\bar L]\cap \eta^*(e^2) = -[N]\cap \eta^*(e)^{k+2}  = -(i_N)_*[N_1]\cap \eta^*(e)^{k+1}=
(i_N)_*\circ(j_1)_*[\bar L_1].
\end{equation}

Let us calculate the element  $(i_L)_*[\bar L_1]$ in the group  $H_{n-4k-4}(\bar L, (\Zt)^{\otimes(2k+2)})$ by an alternative way. Obviously, the manifold  $\bar L_1$ 
is the intersection of the two submanifolds
 $\bar L \cap N_1$
and $\tau(\bar L \cap N_1)$, where $\tau:\bar L\to \bar L $ 
is the involution, which permutes the sheets of the covering  $\bar L\to L$.
From this fact we get in the group
$H_{n-4k-4}(L,(\Zt)^{\otimes (2k+2)})$ the following relation:
 $(i_L)_*[L_1] = [\bar L]\cap (e|_{\bar L} \cup \tau^*(e|_{\bar L}))$, where $e|_L=(\bar\zeta\circ pr_1)^*(e)$.
In a general case the classes $\tau^*(e|_{\bar L})$ and $e|_{\bar L}$ are related unpredictable way. 
But, in a several cases, when the structured group of the normal bundle over  
 $L$ is reduced to a special subgroup, this relation is described explicitly. 

Consider the commutative dyagrame: \\
\centerline{\xymatrix{
\bar L \ar[r]^<<<<<{\bar \zeta}\ar[d]_{\pi} & B(\Z_4\times \Z_4)\rtimes \Z\ar[d]_{\pi} \ar[r]^>>>>>>>{pr_1} &B\Z_4\rtimes \Z\\
L\ar[r]^<<<<<{\zeta} &B((\Z_4\times \Z_4)\rtimes\Z_2)\rtimes \Z}}
By the commutativity, relationship between   $e|_L$  and $\tau^*(e|_L)$
is determined, using  relationship between 
$(pr_1)^*(e)$ and $\tau^*((pr_1)^*(e))$ in the cohomologies of the classifying space
$B(\Z_4\times \Z_4)\rtimes \Z$. This is done in the following two lemmas.

Consider the mappings of the classifying spaces:
$$ I_1:BQ \subset B(\Z_4 \times \Z_4) \rtimes \Z_2 \subset B((\Z_4 \times \Z_4) \rtimes \Z_2) \rtimes \Z,$$
$$ I_2: B(\Z_4 \times \Z_2) \rtimes \Z  \subset B((\Z_4 \times \Z_4) \rtimes \Z_2) \rtimes \Z,$$
which are induced by inclusions onto the subgroups. By
 $\bar I_1$ � $\bar I_2$ let us denote the corresponding mappings of the  2-sheeted coverings 
over this spaces
$$ \bar I_1:B\Z_4\to B(\Z_4 \times \Z_4)  \rtimes \Z,$$
$$\bar I_2: B\Z_4 \times \rtimes \Z  \to  B(\Z_4 \times \Z_4) \rtimes \Z.$$

Put $e_0 = (pr_1)^*(e)$, $e_Q = (\bar I_1)^*(e_0)$, $e_{diag} = (\bar I_2)^*(e_0) $.

\begin{lemma}
The following formula is satisfied: 
  $\tau^*(e_Q)=-e_Q$.
\end{lemma}

\begin{proof}
Let us realized the space $BQ$ as the factor of the quaternion units sphere
   $S^\infty_{\H}/Q = S^\infty_{\H}/\{\pm,1,\pm\i,\pm \j,\pm \k\} $, and the space 
$B\Z_4$ as the quotient of the considered sphere over the cyclic subgroup, generated by the quaternion unit $\i$:
   $S^\infty_{\H}/\Z_4 = S^\infty_{\H}/\{\pm,1,\pm\i\} $. In this case the involution
 $\tau:S^\infty_{\H}/\Z_4\to S^\infty_{\H}/\Z_4$, which permutes the sheets of the covering $B\Z_4\to BQ$,
 corresponds to the multiplication on the quaternion unit $\j$.

It is convenient to investigate the action    $\tau$ not on 
the 2-generator  $i^*(e_Q)$ of the cohomology group of the space  $B\Z_4$, but on the generator 
 $\alpha$ of 1-dimensional homology group, which corresponds to 
 $ i^*(e_Q)$ by the Universal Coefficients Theorem. The generator
 $\alpha$  is represented by the loop  $(\cos t + \i \sin t)x$, where  $t\in[0,\pi/2]$ and $x$ 
is an arbitrary point in
 $S^\infty_{\H}/\{\pm,1,\pm\i\}$. Two such loops determines the same class $\alpha$.
Therefore, the class $\alpha$ represents by the loop $(\cos t + \i \sin t)x$, as well as by the loop $\gamma_(t)=(\cos t + \i \sin t)\j x$, the last representative is more convenient to investigate  the
element $\tau_*(\alpha)$. Because the action  $\tau$ corresponds to the multiplication on  $\j$, the
element $\tau_*(\alpha)$ is represented by the loop $\gamma_2(t)=\j(\cos t + \i \sin t)x$.
Obviously, $\gamma_1(t)$ is given by  $\gamma_2(t)$ by the reversing of the orientation, 
therefore,  $\tau_*(\alpha)=-\alpha$.

\end{proof}

\begin{corollary}\label{cor:Q}
Let an element $\bar x\in H_q(B\Z_4\times\Z_4\rtimes\Z,(\Zt)^{\otimes k})$ 
belongs to the image  $(\bar I_1)_*$. Then for   $\bar x\cap \tau^*(e_0)\in H_{q-2}(B\Z_4\times\Z_4\rtimes\Z,(\Zt)^{\otimes k-1})$ the following equation is satisfied: 
 $\bar x\cap \tau^*(e_0)=-\bar x\cap e_0$.
\end{corollary}

\begin{proof}
Let $\bar x = (\bar I_1)_*(u)$.
Then we have
\begin{multline*}
\bar x \cap\tau^*(e_0) = (\bar I_1)_*(u)\cap\tau^*(e_0) = (\bar I_1)_*(u\cap\tau^*(I^*_1(e_0))) = \\
=(I_1)_*(u\cap\tau^*(e_Q)) = (I_1)_*(u\cap(-e_Q)).
\end{multline*}
By analogous calculations, we get the equation:  $\bar x \cap e_0 = (I_1)_*(u\cap e_Q) $.
\end{proof}

\begin{lemma}
The following equaton is satisfied:  $\tau^*(e_{diag})=e_{diag}$.
\end{lemma}

\begin{proof}
Let us calculate the action $\tau^*$ on the generator of the group  $H^2(B\Z_4\rtimes \Z,\Zt)=\Z_4$.
 Obviously, the covering  $B\Z_4\rtimes\Z \to B \Z_4\times\Z_2 \rtimes\Z $
 is represented as $B\Z_4\rtimes\Z \times S^\infty \stackrel{id\times \pi}\longrightarrow B \Z_4\rtimes\Z \times (S^\infty/\Z_2)$. 
Therefore, the involution $\tau$, which permutes sheets of the double covering, acts in (co)homology of the space  $B\Z_4\rtimes\Z$ by the identity.
 \end{proof}

\begin{corollary}\label{cor:diag}
Assume that $\bar y\in H_q(B\Z_4\times\Z_4\rtimes\Z,(\Zt)^{\otimes k})$ belongs to the image 
of the homomorphism $(\bar I_2)_*$.

(a) Then  for $\bar y\cap \tau^*(e_0)\in H_{q-2}(B\Z_4\times\Z_4\rtimes\Z,(\Zt)^{\otimes k-1})$ the following equality is satisfied:
 $\bar y\cap \tau^*(e_0)=\bar y\cap e_0$

(b) Additionally, if  $\bar y$ belong to the image of the transfer homomorphism of the 2-sheeted covering, then  $\bar y$ is an even element.
\end{corollary}

\begin{proof}
Assume that $\bar y = (\bar I_2)_*(v)$.
Then we get:
\begin{multline*}
\bar y \cap\tau^*(e_0) = (\bar I_2)_*(v)\cap\tau^*(e_0) = (\bar I_2)_*(v\cap\tau^*(I^*_2(e_0))) = \\
=(I_2)_*(u\cap\tau^*(e_{diag})) = (I_2)_*(u\cap e_{diag}).
\end{multline*}
By analogous calculations, we get the equation 
 $\bar y \cap e_0 = (I_2)_*(v\cap e_{diag}) $.

To prove the statement (b), let us note, that 
the factor-homomorphism $\Z_4\times\Z_2 \rtimes\Z \to\Z_4 \rtimes\Z$ determines the mapping
 $s:B\Z_4\times\Z_2 \rtimes\Z \to B\Z_4 \rtimes\Z$, for which the composition $s\circ \pi$ 
is homotopic to the identity.
 Then $\bar y = s_* (\pi_* (\bar y))$.   The element $s_*\pi_* (\bar y) $ is even.
 \end{proof}

\subsection*{An application}
\bigskip

\noindent
{\bf Additional assumption}
Let us assume that the image
$\zeta_{\ast}([L^{n-4k}])$ is in the subgroup 
$$ \mathrm{Im} \;(I_1)_*  + \mathrm{Im} \;(I_2)_* \subset  H_{n-4k}(B((\Z_4 \times \Z_4) \rtimes \Z_2) \rtimes \Z ;\Z)). $$

\bigskip

\begin{theorem}\label{3}
Assume that $n\equiv 3\pmod 4$, $k$ is an even number, 	$n-4k \ge 7$.
Let us assume that the additional condition above is satisfied:
Then $\eta_{\ast}([N^{n-2k}] \in H_{n-2k}(B\Z/4 \rtimes \Z;\Z) = \Z_4$ is an even element.
\end{theorem}


\begin{proof}
Let $\zeta_{\ast}([L^{n-4k}]) = x+y$,
where $x=  (I_1)_*(u) $ and $y=  (I_2)_*(v)$ for suitable elements  $u $  � $v$. 
Using the transfer homomorphism, we get:
$\bar\zeta_{\ast}([\bar L^{n-4k}]) = \bar x+\bar y$, where  $\bar x=  (\bar I_1)_*(\bar u) $ � $\bar y=  (\bar I_2)_*(\bar v)$.

Let us consider  the following commutative diagram: \\
\centerline{\xymatrix{
\bar L_1\ar[r]^<<<<<<<<<<{j_1}\ar[d]_{i_L}& N_1\ar[d]_{i_N}\\
\bar L  \ar[r]^<<<<<<<<<<{j}\ar[dr]_<<<<<<<{\bar\zeta}& N \ar[r]^{\eta}&B\Z_4\rtimes\Z\\
&B(\Z_4\times\Z_4)\rtimes\Z\ar[ru]_{pr_1}}}
and calculate the image $[L_1]\in H_{n-4k-4}(\bar L_1,\Z)$
in the group $H_{n-4k-4}(B\Z_4\rtimes \Z,\Z)=\Z_4$ by two different way, using the diagram. 

At the first step, the image is given by the formula: $(\eta\circ i_N\circ j_1)_*[\bar L_1]$.
Let us apply to the both sides of the formula~(\ref{eq:vazhnoe}) the homomorphism  $\eta_*$, we get 
\begin{equation}\label{eq:raz}
(\eta\circ i_N\circ j_1)_*[\bar L_1]= (\eta\circ j)_*[\bar L]\cap e^2 = (pr_1)_*(\zeta_*[\bar L])\cap e^2  = (pr_1)_*(\bar x +\bar y)\cap e^2.
\end{equation}
At the second step, let us represent this element as follows:
$(pr_1\circ \bar\zeta\circ i_L)_*[\bar L_1]$.  Using corollaries
~\ref{cor:Q} and~\ref{cor:diag}(a), we get:
 \begin{equation}\label{eq:dva}
 (pr_1\circ \bar\zeta\circ i_L)_*[\bar L_1]= (pr_1)_*\circ (\bar\zeta)_*([\bar L]\cap (\tau^*(e|_L)\cup e|_L)) = 
 (pr_1)_*(\bar \zeta_*[\bar L]\cap (\tau^(e_0)\cup e_0)) = (pr_1)_*(-\bar x+\bar y )\cap e^2.
\end{equation}
The results of calculations  (\ref{eq:raz}) and (\ref{eq:dva}) coincide, then we get 
$2 (pr_1)_*(\bar x)\cap e^2=0$ in the group $H_{n-4k-4}(B\Z_4\rtimes \Z,\Z)=\Z_4$,
therefore the element $(pr_1)_*(\bar x)\cap e^2$ is even.  By corollary~\ref{cor:diag}(b)
the element  $\bar y$ is even, this proves that the element 
$(pr_1)_*(\zeta_*[\bar L])\cap e^2 = \eta_*(j_*[\bar L])\cap e^2$ is even.

From the Herbert Formula for the immersion  $g$ we get: $ \eta_*(j_*[\bar L])\cap e^2 = -\eta_*([N])\cap e^{k+2}$. Because the cap-product with the Euler class $-\cap e$ is an isomorphism, the element 
$\eta_*([N])$ is even. 
\end{proof}
The authors were supported in part by RFBR Grant No 15-01-06302.


\begin{thebibliography}{99}

\bibitem{herbert} Herbert R.J. {\em Multiple points of immersed manifolds}. Providence (RI): AMS, 1981 (Mem. AMS; v. 250)

\bibitem{eccles-grant} P.J.Eccles, M.Grant {\em Bordism classes represented by miltiple point manifolds of immersed manifolds}, ����� ��������������� ��������� ��. �.�. ��������, 2006, �.252, �. 55-60 

\bibitem{lippner-szucs} �. �������, �. ���, {\it ����� ��������������
������� �������� ������� �����}, ��������������� � ���������� ����������, 2005, ��� 11, N 5, �. 107�116.

\bibitem{A-E}
P.M.Akhmet'ev, P.J.Eccles.
The relationship between framed bordism and skew-framed bordism //
Bull. Lond. Math. Soc. 39, No. 3, 473-481 (2007).


\bibitem{A-F}
P.Akhmetiev, O.Frolkina. 
On non-immersibility of $\RP^{10}$ to $\R^{15}$ //
Topology and its Applications, V.160, 11 (2013). 1241-1254.


\bibitem{akhmetiev-frolkina} �.�.��������, �.�.�������� {\it � ��������� ����� ������������ ����������� ����������}, ��������������� � ���������� ����������, 2016, 21 (to appear)



\bibitem{Eccles}
P.J.Eccles.
Multiple points of codimension one immersions.
Topology Symposium, Siegen 1979, Lecture Notes in Math., vol. 788, Springer, Berlin, 1980, pp. 23�38.



\bibitem{Grant}
M.Grant.
Bordism of Immersions. Thesis. 2006.

\bibitem{Hirsch-paper}
M.W.Hirsch.
Immersions of manifolds.
Trans. Amer. Math. Soc. 93 (1959), 242--276.

\bibitem{pontr}
L.S.Pontryagin.
Smooth manifolds and their applications to homotopy theory. 
Transl., II. Ser., Am. Math. Soc. 11, 1-114 (1959).
Translated from: 
Trudy Mat. Inst. im. Steklov. no. 45. Izdat. Akad. Nauk SSSR, Moscow, 1955.


\bibitem{Sz2}
A.Sz$\ddot{\rm{u}}$cs, {\it Cobordism of immersions and singular maps, loop
spaces and multiple points}, Geometric and Algebraic Topology,
Banach Center Publ., 18 PWN, Warsaw (1986) 239-253.

\bibitem{Wells}
R.Wells, {\it Cobordism groups of immersions.} // Topology 5 (1966) 281-
294.

\end{thebibliography}
\end{document}